 \documentclass[a4paper,12pt]{article}
    \usepackage[top=2.5cm,bottom=2.5cm,left=2.5cm,right=2.5cm]{geometry}
    \usepackage{cite, amsmath, amssymb}
    \usepackage[margin=1cm,%
                font=small,%
                format=hang,%
                labelsep=period,%
                labelfont=bf]{caption}
\usepackage[latin1]{inputenc}
\usepackage{amsmath}
\usepackage{amsfonts}
\usepackage{amssymb}
\usepackage{cite}
\usepackage{amsthm}
\usepackage{makeidx}
\usepackage{graphicx}
\usepackage{mathrsfs,xcolor}
\usepackage{enumerate}
\usepackage{blkarray}
\usepackage{bm}
\usepackage{setspace}
\usepackage{kbordermatrix}
\usepackage{float}
\usepackage{listings}
\usepackage{authblk}

\numberwithin{table}{section}
\numberwithin{equation}{section}

\theoremstyle{plain}
\newtheorem{theorem}{Theorem}[section]
\newtheorem{proposition}[theorem]{Proposition}
\newtheorem{definition}[theorem]{Definition}
\newtheorem{lemma}[theorem]{Lemma}
\newtheorem{example}[theorem]{Example}

\newtheorem{corollary}[theorem]{Corollary}
\newtheorem{remark}[theorem]{Remark}

\usepackage{authblk}
\author[1,2,*]{ \textbf{Bryan S. Hernandez}}

\affil[1]{\small \textit{Institute of Mathematics, University of the Philippines Diliman, Quezon City 1101, Philippines}}
\affil[2]{\small \textit{Institute of Mathematical Sciences and Physics, University of the Philippines,  Los Ba\~{n}os, Laguna 4031, Philippines
}}
\affil[*]{Corresponding author: \texttt{bryan.hernandez@upd.edu.ph}}

\title{\vspace{3.5cm}\textbf{On the Independence of Fundamental Decompositions of Power-Law Kinetic Systems}}

\date{\normalsize January 2020}

\begin{document}
\maketitle
\thispagestyle{empty}
\begin{abstract}
The fundamental decomposition of a chemical reaction network (CRN) is induced by partitioning the reaction set into ``fundamental classes''. It was the basis of the Higher Deficiency Algorithm for mass action systems of Ji and Feinberg, and the Multistationarity Algorithm for power-law kinetic systems of Hernandez et al.
In addition to our previous work, we provide important properties of the independence (i.e., the network's stoichiometric subspace is the direct sum of the subnetworks' stoichiometric subspaces) and the incidence-independence (i.e., the image of the network's incidence map is the direct sum of the incidence maps' images of the subnetworks) of these decompositions.
Feinberg established the essential relationship between independent decompositions and the set of positive equilibria of a network, which we call the Feinberg Decomposition Theorem (FDT). Moreover, Fari{\~n}as et al. recently documented its version for incidence-independence.
Fundamental decomposition divides the network into subnetworks of deficiency either 0 or 1 only. Hence, available results for lower deficiency networks, such as the Deficiency Zero Theorem (DZT), can be used. These justify the study of independent fundamental decompositions. A MATLAB program which (i) computes the subnetworks of a CRN under the fundamental decomposition and (ii) is useful for determining whether the decomposition is independent and incidence-independent is also created. Finally, we provide the following solution for determining multistationarity of CRNs with the following steps: (1) the use of the program, (2) the application of available results for CRNs with deficiency 0 or 1 (e.g., DZT), and (3) the use of FDT.
We illustrate the solution by showing that the generalization of a subnetwork of Schmitz's carbon cycle model by Hernandez et al., endowed with mass action kinetics, has no capacity for multistationarity.
\end{abstract}
\baselineskip=0.30in

\section{Introduction} 
\label{intro}
The fundamental decomposition (or ``$\mathscr{F}$-decomposition'') of a chemical reaction network (CRN) is the set of subnetworks generated by the partition of its set of reactions into ``fundamental classes''. This was introduced by Ji and Feinberg \cite{ji} in 2011 as the basis of their Higher Deficiency Algorithm (HDA) for mass action systems. In 2020, Hernandez et al. \cite{hernandez} extended the HDA to power-law kinetic systems with reactant-determined interactions (PL-RDK). These are reactions branching from the same reactant complex having identical kinetic order vectors. By combining this extension with a method that transforms a power-law kinetic system with non-reactant-determined interactions (PL-NDK) to a dynamically equivalent PL-RDK system, the Multistationarity Algorithm (MSA) was established and used to determine multistationarity (i.e., the system admits at least two equilibria) of any power-law kinetic system within a stoichiometric class.

This paper explores properties of the $\mathscr{F}$-decomposition, in particular, its
independence (i.e., the network's stoichiometric subspace is the direct sum of the subnetworks' stoichiometric subspaces) and its incidence-independence (i.e., the image of the network's incidence map is the direct sum of the incidence maps' images of the subnetworks). In our previous work \cite{hernandez2}, we mentioned that for independent fundamental decompositions, the transformation that converts a PL-NDK system to a PL-RDK system is not necessary. This was actually one of the main motivations of the study of independent fundamental decompositions.

M. Feinberg established the essential relationship between independent decompositions and the set of positive equilibria of a network in 1987, which we call the Feinberg Decomposition Theorem (FDT) \cite{feinberg12}. A corresponding relationship between incidence-independent, weakly reversible decompositions and complex-balanced equilibria of a weakly reversible network was recently documented by Fari{\~n}as et al. \cite{FML2018}. These also justify the relevance of independent decompositions. Moreover, fundamental decompositions decompose CRNs into subnetworks where each of the subnetworks has deficiency either 0 or 1. Thus, the well-known Deficiency Zero and Deficiency One Theorems, and other lower deficiency theorems in literature, maybe used. Hence, by determining whether the fundamental decomposition of a CRN is independent, one maybe able to decide whether it has the capacity for multistationarity using the lower deficiency theorems and the FDT.

The paper is organized as follows: Section \ref{prelim} provides the fundamentals of chemical reaction networks and chemical kinetic systems. It also presents relevant results from the decomposition theory of CRNs. Section \ref{orientations:decompositions} collects important properties of the fundamental and other related decompositions.
In Section \ref{funde_decomposition}, we provide results between the independence and the incidence-independence of the fundamental decomposition and the types of the networks.
We also establish the following steps in determining whether a CRN has the capacity for multistationarity:
\begin{itemize}
\item[1.] Is the fundamental decomposition of the network independent? If yes, we go to the next step.
\item[2.] Under the fundamental decomposition, each subnetwork has a deficiency either 0 or 1. Use the existing results for such CRNs.
\item[3.] Use the FDT. 
\end{itemize}
In addition, the solution of the multistationarity of the CRN of the
generalization of a subnetwork of Schmitz's carbon cycle model by Hernandez et al. endowed with mass action kinetics is also given. This serves as the major example for this section.
In Section \ref{decomposition_CFRM}, we establish relationships between the fundamental decomposition of a kinetic system and its transform under the CF-RM.  Conclusions and an outlook constitute Section \ref{sec:conclusion}. Tables of acronyms and frequently used symbols are provided in Appendix \ref{nomenclature:appendix}. Finally, Appendix \ref{matlabprogram} provides a program that outputs the subnetworks under the fundamental decomposition, and determines whether the decomposition is independent or incidence-independent.

\section{Fundamentals of Chemical Reaction Networks and Kinetic Systems}
\label{prelim}
\indent In this section, we recall some fundamental notions about chemical reaction networks and chemical kinetic systems \cite{arc_jose,feinberg}. We also present important preliminaries on the decomposition theory which was introduced by Feinberg in \cite{feinberg12}.

\subsection{Fundamentals of Chemical Reaction Networks}

\begin{definition}
A {\bf chemical reaction network} (CRN) $\mathscr{N}$ is a triple $\left(\mathscr{S},\mathscr{C},\mathscr{R}\right)$ of nonempty finite sets where $\mathscr{S}$, $\mathscr{C}$, and $\mathscr{R}$ are the sets of $m$ species, $n$ complexes, and $r$ reactions, respectively, such that
$\left( {{C_i},{C_i}} \right) \notin \mathscr{R}$ for each $C_i \in \mathscr{C}$; and
for each $C_i \in \mathscr{C}$, there exists $C_j \in \mathscr{C}$ such that $\left( {{C_i},{C_j}} \right) \in \mathscr{R}$ or $\left( {{C_j},{C_i}} \right) \in \mathscr{R}$.
\end{definition}

\begin{definition}
The {\bf molecularity matrix}, denoted by $Y$, is an $m\times n$ matrix such that $Y_{ij}$ is the stoichiometric coefficient of species $X_i$ in complex $C_j$.
The {\bf incidence matrix}, denoted by $I_a$, is an $n\times r$ matrix such that 
$${\left( {{I_a}} \right)_{ij}} = \left\{ \begin{array}{rl}
 - 1&{\rm{ if \ }}{C_i}{\rm{ \ is \ in \ the\ reactant \ complex \ of \ reaction \ }}{R_j},\\
 1&{\rm{  if \ }}{C_i}{\rm{ \ is \ in \ the\ product \ complex \ of \ reaction \ }}{R_j},\\
0&{\rm{    otherwise}}.
\end{array} \right.$$
The {\bf stoichiometric matrix}, denoted by $N$, is the $m\times r$ matrix given by 
$N=YI_a$.
\end{definition}
Let $\mathscr{I}=\mathscr{S}, \mathscr{C}$ or $\mathscr{R}$. We denote the standard basis for $\mathbb{R}^\mathscr{I}$ by $\left\lbrace \omega_i \in \mathbb{R}^\mathscr{I} \mid i \in \mathscr{I} \right\rbrace$.

\begin{definition}
Let $\mathscr{N}=(\mathscr{S,C,R})$ be a CRN. The {\bf{incidence map}} $I_a : \mathbb{R}^\mathscr{R} \rightarrow \mathbb{R}^\mathscr{C}$ is the linear map such that for each reaction $r:C_i \rightarrow C_j \in \mathscr{R}$, the basis vector $\omega_r$ to the vector $\omega_{C_j}-\omega_{C_i} \in \mathscr{C}$.
\end{definition} 

\begin{definition}
The {\bf reaction vectors} for a given reaction network $\left(\mathscr{S},\mathscr{C},\mathscr{R}\right)$ are the elements of the set $\left\{{C_j} - {C_i} \in \mathbb{R}^\mathscr{S}|\left( {{C_i},{C_j}} \right) \in \mathscr{R}\right\}.$
\end{definition}

\begin{definition}
The {\bf stoichiometric subspace} of a reaction network $\left(\mathscr{S},\mathscr{C},\mathscr{R}\right)$, denoted by $S$, is the linear subspace of $\mathbb{R}^\mathscr{S}$ given by $S = span\left\{ {{C_j} - {C_i} \in \mathbb{R}^\mathscr{S}|\left( {{C_i},{C_j}} \right) \in \mathscr{R}} \right\}.$ The {\bf rank} of the network, denoted by $s$, is given by $s=\dim S$. The set $\left( {x + S} \right) \cap \mathbb{R}_{ \ge 0}^\mathscr{S}$ is said to be a {\bf stoichiometric compatibility class} of $x \in \mathbb{R}_{ \ge 0}^\mathscr{S}$.
\end{definition}

\begin{definition}
Two vectors $x, x^{*} \in {\mathbb{R}^\mathscr{S}}$ are {\bf stoichiometrically compatible} if $x-x^{*}$ is an element of the stoichiometric subspace $S$.
\end{definition}

We can view complexes as vertices and reactions as edges. With this, CRNs can be seen as graphs. At this point, if we are talking about geometric properties, {\bf vertices} are complexes and {\bf edges} are reactions. If there is a path between two vertices $C_i$ and $C_j$, then they are said to be {\bf connected}. If there is a directed path from vertex $C_i$ to vertex $C_j$ and vice versa, then they are said to be {\bf strongly connected}. If any two vertices of a subgraph are {\bf (strongly) connected}, then the subgraph is said to be a {\bf (strongly) connected component}. The (strong) connected components are precisely the {\bf (strong) linkage classes} of a CRN. The maximal strongly connected subgraphs where there are no edges from a complex in the subgraph to a complex outside the subgraph is said to be the {\bf terminal strong linkage classes}.
We denote the number of linkage classes and the number of strong linkage classes by $l$ and $sl$, respectively.
A CRN is said to be {\bf weakly reversible} if $sl=l$.

\begin{definition}
For a CRN, the {\bf deficiency} is given by $\delta=n-l-s$ where $n$ is the number of complexes, $l$ is the number of linkage classes, and $s$ is the dimension of the stoichiometric subspace $S$.
\end{definition}

\subsection{Fundamentals of Chemical Kinetic Systems}

\begin{definition}
A {\bf kinetics} $K$ for a reaction network $\left(\mathscr{S},\mathscr{C},\mathscr{R}\right)$ is an assignment to each reaction $r: y \to y' \in \mathscr{R}$ of a rate function ${K_r}:{\Omega _K} \to {\mathbb{R}_{ \ge 0}}$ such that $\mathbb{R}_{ > 0}^\mathscr{S} \subseteq {\Omega _K} \subseteq \mathbb{R}_{ \ge 0}^\mathscr{S}$, $c \wedge d \in {\Omega _K}$ if $c,d \in {\Omega _K}$, and ${K_r}\left( c \right) \ge 0$ for each $c \in {\Omega _K}$.
Furthermore, it satisfies the positivity property: supp $y$ $\subset$ supp $c$ if and only if $K_r(c)>0$.
The system $\left(\mathscr{S},\mathscr{C},\mathscr{R},K\right)$ is called a {\bf chemical kinetic system}.
\end{definition}

\begin{definition}
The {\bf species formation rate function} (SFRF) of a chemical kinetic system is given by $f\left( x \right) = NK(x)= \displaystyle \sum\limits_{{C_i} \to {C_j} \in \mathscr{R}} {{K_{{C_i} \to {C_j}}}\left( x \right)\left( {{C_j} - {C_i}} \right)}.$
\end{definition}
The ordinary differential equation (ODE) or dynamical system of a chemical kinetics system is $\dfrac{{dx}}{{dt}} = f\left( x \right)$. An {\bf equilibrium} or {\bf steady state} is a zero of $f$.

\begin{definition}
The {\bf set of positive equilibria} of a chemical kinetic system $\left(\mathscr{S},\mathscr{C},\mathscr{R},K\right)$ is given by ${E_ + }\left(\mathscr{S},\mathscr{C},\mathscr{R},K\right)= \left\{ {x \in \mathbb{R}^\mathscr{S}_{>0}|f\left( x \right) = 0} \right\}.$
\end{definition}

A CRN is said to admit {\bf multiple equilibria} if there exist positive rate constants such that the ODE system admits more than one stoichiometrically compatible equilibria.

\begin{definition}
A kinetics $K$ is {\bf complex factorizable} if, for $K(x) = k I_K(x)$,  the interaction map $I_K : \mathbb{R}^\mathscr{S} \to \mathbb{R}^\mathscr{R}$
factorizes via the space of complexes $\mathbb{R}^\mathscr{C} : I_K = I_k \circ \psi _K$ with  $\psi _K: \mathbb{R}^\mathscr{S} \to \mathbb{R}^\mathscr{C}$ as factor map and
$I_k = diag(k) \circ \rho'$ with $\rho' : \mathbb{R}^\mathscr{C} \to \mathbb{R}^\mathscr{R}$ assigning the value at a reactant complex to all its reactions. 

\end{definition}

\begin{definition}
A kinetics $K$ is a {\bf power-law kinetics} (PLK) if 
${K_i}\left( x \right) = {k_i}{{x^{{F_{i}}}}} $ for $i =1,...,r$ where ${k_i} \in {\mathbb{R}_{ > 0}}$ and ${F_{ij}} \in {\mathbb{R}}$. The power-law kinetics is defined by an $r \times m$ matrix $F$, called the {\bf kinetic order matrix} and a vector $k \in \mathbb{R}^\mathscr{R}$, called the {\bf rate vector}.
\end{definition}
If the kinetic order matrix is the transpose of the molecularity matrix, then the system becomes the well-known {\bf mass action kinetics (MAK)}.

\begin{definition}
A PLK system has {\bf reactant-determined kinetics} (of type PL-RDK) if for any two reactions $i, j$ with identical reactant complexes, the corresponding rows of kinetic orders in $F$ are identical, i.e., ${f_{ik}} = {f_{jk}}$ for $k = 1,2,...,m$. A PLK system has {\bf non-reactant-determined kinetics} (of type PL-NDK) if there exist two reactions with the same reactant complexes whose corresponding rows in $F$ are not identical.
\end{definition}

We now state the Deficiency Zero and Deficiency One Theorems by Feinberg \cite{feinberg,feinberg12,feinberg2}.

\begin{theorem} (Deficiency Zero Theorem)
For any CRN of deficiency zero, the following statements hold:
\begin{enumerate}
\item[i.] If the network is not weakly reversible, then for arbitrary kinetics, the differential equations for the corresponding reaction system cannot admit
a equilibrium.
\item[ii.] If the network is not weakly reversible, then for arbitrary kinetics, the differential equations for the corresponding reaction system cannot admit
a cyclic composition trajectory containing a positive composition.
\item[iii.] If the network is weakly reversible, then for any mass action kinetics (but regardless
of the positive values the rate constants take), the resulting differential equations have the
following properties:\\
There exists within each positive stoichiometric compatibility class precisely one equilibrium; that equilibrium is asymptotically stable; there is no nontrivial cyclic composition
trajectory along which all species concentrations are positive.
\end{enumerate}
\end{theorem}

\begin{theorem} (Deficiency One Theorem)
Consider a mass action system for which
the underlying reaction network has l linkage classes. Let $\delta$ be the deficiency of the network
and let $\delta_\theta$ be the deficiency of the $\theta$th linkage class. Suppose that the following
conditions hold:
\begin{enumerate}
\item[i.] $\delta_\theta \le 1$ for each linkage class and
\item[ii.] the sum of the deficiencies of all the individual linkage classes is the deficiency of the whole network.
\end{enumerate}
If the network is weakly reversible, then for any mass action kinetics, the differential equations for the system admit
precisely one equilibrium in each positive stoichiometric compatibility class.
\end{theorem}

\subsection{Review of Decomposition Theory}
\label{sect:decomposition}

We now recall some definitions and earlier results from the decomposition theory of chemical reaction networks. 

\begin{definition}
A {\bf decomposition} of $\mathscr{N}$ is a set of subnetworks $\{\mathscr{N}_1, \mathscr{N}_2,...,\mathscr{N}_k\}$ of $\mathscr{N}$ induced by a partition $\{\mathscr{R}_1, \mathscr{R}_2,...,\mathscr{R}_k\}$ of its reaction set $\mathscr{R}$. 
\end{definition}

We denote a decomposition with 
$\mathscr{N} = \mathscr{N}_1 \cup \mathscr{N}_2 \cup ... \cup \mathscr{N}_k$
since $\mathscr{N}$ is a union of the subnetworks in the sense of \cite{GHMS2018}. It also follows immediately that, for the corresponding stoichiometric subspaces, 
${S} = {S}_1 + {S}_2 + ... + {S}_k$. It is also useful to consider refinements and coarsenings of decompositions.

\begin{definition}
A network decomposition $\mathscr{N} = \mathscr{N}_1 \cup \mathscr{N}_2 \cup ... \cup \mathscr{N}_k$  is a {\bf refinement} of
$\mathscr{N} = {\mathscr{N}'}_1 \cup {\mathscr{N}'}_2 \cup ... \cup {\mathscr{N}'}_{k'}$ 
(and the latter a coarsening of the former) if it is induced by a refinement  
$\{\mathscr{R}_1, \mathscr{R}_2,...,\mathscr{R}_k\}$
of $\{{\mathscr{R}'}_1 \cup {\mathscr{R}'}_2 \cup ... \cup {\mathscr{R}'}_{k'}\}$, i.e., 
each ${\mathscr{R}}_i$ is contained in an ${\mathscr{R}'}_j$. 
\end{definition}

In \cite{feinberg12}, Feinberg introduced the important concept of independent decomposition.

\begin{definition}
A network decomposition $\mathscr{N} = \mathscr{N}_1 \cup \mathscr{N}_2 \cup ... \cup \mathscr{N}_k$  is {\bf independent} if its stoichiometric subspace is a direct sum of the subnetwork stoichiometric subspaces.
\end{definition}

For any decomposition, it also holds that 
$\text{Im } I_a = \text{Im } I_{a,1} + ... + \text{Im } I_{a,k}$, where $$\text{Im } I_{a,i} = I_a(\mathbb{R}^{{\mathscr{R}}_i}).$$

\begin{definition}
A network decomposition is {\bf incidence-independent} if ${\rm{Im \ }} I_a$ is a direct sum of the ${\rm{Im \ }} I_{a,i}$. It is {\bf bi-independent} if it is both independent and incidence-independent.
\end{definition}

An equivalent formulation of showing incidence-independent is to satisfy $n - l = \sum {\left( {{n_i} - {l_i}} \right)}$, where $n_i$ is the number of complexes and $l_i$ is the number of linkage classes, in each subnetwork $i$.
It was shown that for independent decompositions, $\delta \le \delta_1 +\delta_2 ... +\delta_k$ \cite{fortun2}. On the other hand, for incidence-independent decompositions, $\delta \ge \delta_1 +\delta_2 ... +\delta_k$ \cite{FML2018}. 

Feinberg established the following relation between an independent decomposition and the set of positive equilibria of a kinetics on the network.

\begin{theorem} (Feinberg Decomposition Theorem \cite{feinberg12})
\label{feinberg:decom:thm}
Let $P(\mathscr{R})=\{\mathscr{R}_1, \mathscr{R}_2,...,\mathscr{R}_k\}$ be a partition of a CRN $\mathscr{N}$ and let $K$ be a kinetics on $\mathscr{N}$. If $\mathscr{N} = \mathscr{N}_1 \cup \mathscr{N}_2 \cup ... \cup \mathscr{N}_k$ is the network decomposition of $P(\mathscr{R})$ and ${E_ + }\left(\mathscr{N}_i,{K}_i\right)= \left\{ {x \in \mathbb{R}^\mathscr{S}_{>0}|N_iK_i(x) = 0} \right\}$ then
\[{E_ + }\left(\mathscr{N}_1,K_1\right) \cap {E_ + }\left(\mathscr{N}_2,K_2\right) \cap ... \cap {E_ + }\left(\mathscr{N}_k,K_k\right) \subseteq  {E_ + }\left(\mathscr{N},K\right).\]
If the network decomposition is independent, then equality holds.
\end{theorem}

The following theorem is the analogue of Feinberg's result for incidence-independent decompositions and complex-balanced equilibria \cite{FML2018}.

\begin{theorem}
\label{decomposition:thm:2}
Let $\mathscr{N}$ be a network, $K$ any kinetics and $\mathscr{N} = \mathscr{N}_1 \cup \mathscr{N}_2 \cup ... \cup \mathscr{N}_k$ an incidence-independent decomposition of weakly reversible subnetworks. Then $\mathscr{N}$ is weakly reversible and 
\begin{itemize}
\item[i.] ${Z_ + }\left( {\mathscr{N} ,K} \right) =  \cap {Z_ + }\left( {{\mathscr{N} _i},K} \right)$ for each subnetwork ${\mathscr{N} _i}$.
\item[ii.] If ${Z_ + }\left( {\mathscr{N},K} \right) \ne \emptyset $ then ${Z_ + }\left( {{\mathscr{N}_i},K} \right) \ne \emptyset $ for each subnetwork ${\mathscr{N} _i}$.
\item[iii.] If the decomposition is a $\mathscr{C}$-decomposition and $K$ a complex factorizable kinetics then  ${Z_ + }\left( {{\mathscr{N}_i},K} \right) \ne \emptyset $ for each subnetwork ${\mathscr{N} _i}$ implies that ${Z_ + }\left( {{\mathscr{N}},K} \right) \ne \emptyset $.
\end{itemize}
\end{theorem}

\section{The $\mathscr{O}$-, $\mathscr{P}$-, and $\mathscr{F}$-decompositions of CRNs}
\label{orientations:decompositions}
We review the concepts and properties underlying HDA and its extension to PL-RDK systems in the context of decomposition theory \cite{hernandez,ji}.

\begin{definition} A subset $\mathscr{O}$ of $\mathscr{R}$ is said to be an {\bf orientation} if for every reaction $y \to y' \in \mathscr{R}$, either $y \to y' \in \mathscr{O}$ or $y' \to y \in \mathscr{O}$, but not both.
\end{definition}

For an orientation $\mathscr{O}$, we define a linear map ${L_\mathscr{O}}:{\mathbb {R}^\mathscr{O}} \to S$ such that
\[{L_\mathscr{O}}(\alpha)= \sum\limits_{y \to y' \in \mathscr{O}} {{\alpha _{y \to y'}}\left( {y' - y} \right)}.\]

Each orientation $\mathscr{O}$ defines a partition of $\mathscr{N}$ into $\mathscr{O}$ and its complement $\mathscr{O}'$, which generates the following decomposition:

\begin{definition}
For an orientation $\mathscr{O}$ on $\mathscr{N}$, the {\bf $\mathscr{O}$-decomposition} of $\mathscr{N}$ consists of the subnetworks $\mathscr{N}_\mathscr{O}$ and $\mathscr{N}_{\mathscr{O}'}$, i.e., $\mathscr{N}=\mathscr{N}_\mathscr{O} \cup \mathscr{N}_\mathscr{O'}$.
\end{definition}

We now review the important concept of ``equivalence classes'' from \cite{ji}.
Let $\left\{ {{v^l}} \right\}_{l = 1}^d$ be a basis for $Ker{L_\mathscr{O}}$.
If for $y \to y' \in \mathscr{O}$, $v_{y \to y'}^l =0$ for all $1 \le l \le d$ then the reaction $y \to y' $ belongs to the zeroth equivalence class $P_0$.
For $y \to y', {\overline y  \to \overline y '} \in \mathscr{O} \backslash P_0$, if there exists $\alpha \ne 0$ such that $v_{y \to y'}^l = \alpha v_{\overline y  \to \overline y '}^l$ for all $1 \le l \le d$, then the two reactions are in the same equivalence class denoted by $P_i$, $i \ne 0$.

The central concept of ``fundamental classes'' is actually the basis of the Higher Deficiency Algorithm of Ji and Feinberg. The reactions $y \to y'$ and $\overline y  \to \overline y '$ in $\mathscr{R}$ belong to the same {\bf fundamental class} if at least one of the following is satisfied \cite{ji}.
\begin{itemize}
\item[i.] $y \to y'$ and $\overline y  \to \overline y '$ are the same reaction.
\item[ii.] $y \to y'$ and $\overline y  \to \overline y '$ are reversible pair.
\item[iii.] Either $y \to y'$ or $y' \to y$, and either $\overline y  \to \overline y '$ or $\overline y'  \to \overline y $ are in the same equivalence class on $\mathscr{O}$.
\end{itemize}

It is worth mentioning that he orientation $\mathscr{O}$ is partitioned into equivalence classes while the reaction set $\mathscr{R}$ is partitioned into fundamental classes.

\begin{definition}
The {\bf $\mathscr{F}$-decomposition} of $\mathscr{N}$ is the decomposition generated by the partition of $\mathscr{R}$ into fundamental classes.
\end{definition}

\begin{theorem} \cite{hernandez2}
\label{PifC}
Let $\mathscr{N}_\mathscr{O}$ be the subnetwork of  $\mathscr{N}$ defined by the orientation $\mathscr{O}$ being a subset of $\mathscr{R}$. Then the following holds:
\begin{itemize}
\item [i.] The $\mathscr{P}$-decomposition of $\mathscr{N}_\mathscr{O}$ is independent if and only if the $\mathscr{F}$-decomposition of $\mathscr{N}$ is independent.
\item [ii.] The $\mathscr{P}$-decomposition of $\mathscr{N}_\mathscr{O}$ is incidence-independent if and only if the $\mathscr{F}$-decomposition of $\mathscr{N}$ is incidence-independent.
\item [iii.] The $\mathscr{P}$-decomposition of $\mathscr{N}_\mathscr{O}$ is bi-independent if and only if the $\mathscr{F}$-decomposition of $\mathscr{N}$ is bi-independent.
\end{itemize}
\end{theorem}

We denote the zeroth equivalence class as $P_0$, the nontrivial equivalence classes as $P_1,P_2,...,P_{w_1}$ and the trivial equivalence classes as $P_{w_1+1},P_{w_1+2},...,$ $P_{w_1+w_2}$. If the zeroth equivalence class is nonempty, then there are $w_1+w_2+1=w+1$ equivalence classes where $w=w_1+w_2$.

We then get the sets $\left( {{C_i}\backslash {P_i}} \right)$ for $i=0,1,2,...,w$. Let $t$ be the number of nonempty sets excluding $P_0$.
If $P_0$ is nonempty, we define $\widetilde {\mathscr{P}}$-decomposition in the following manner: 
\[\left( {\bigcup\limits_{i = 0}^w {{P_i}} } \right) \cup \left( {\bigcup\limits_{i = 0}^t {\left( {{C_i}\backslash {P_i}} \right)} } \right).\]

We now establish our basic new results for the relationship of $\widetilde {\mathscr{P}}$- and ${\mathscr{F}}$-decompositions.

\begin{proposition}
The $\widetilde {\mathscr{P}}$-decomposition is a refinement of the ${\mathscr{F}}$-decomposition.
\end{proposition}

\begin{proof}
By definition, $\widetilde {\mathscr{P}}$-decomposition induces a partition of the reaction set of the whole network. For each $i$, $C_i$ is partitioned by $P_i$ and $C_i \backslash P_i$.
\end{proof}

\begin{proposition}
The $\widetilde {\mathscr{P}}$-decomposition is a refinement of the ${\mathscr{O}}$-decomposition.
\end{proposition}

\begin{proof}
This follows from the fact that $\widetilde {\mathscr{O}}$-decomposition induces a partition of the reaction set of the whole network.
\end{proof}

\begin{proposition}
If a CRN is composed of irreversible reactions only, then
\begin{itemize}
\item [i.] $\widetilde {\mathscr{P}}$-decomposition is independent if and only if ${\mathscr{F}}$-decomposition is independent;
\item [ii.] $\widetilde {\mathscr{P}}$-decomposition is incidence-independent if and only if ${\mathscr{F}}$-decomposition is incidence-independent; and,
\item [iii.] $\widetilde {\mathscr{P}}$-decomposition is bi-independent if and only if ${\mathscr{F}}$-decomposition is bi-independent.
\end{itemize}
\end{proposition}

\section{On Types of Fundamental Decompositions of CRNs}
\label{funde_decomposition}

We begin this section with the classification of the subnetworks occurring in a $\mathscr{P}$-decomposition of a network into 3 types. Note that the subnetworks from the decomposition have deficiency either 0 or 1 only.

\begin{lemma} {\bf \cite{ji}}
Let $\mathscr{N}=\left(\mathscr{S},\mathscr{C},\mathscr{R}\right)$ be a CRN and $\mathscr{O}$ be an orientation. Let $\mathscr{{N}}_{\mathscr{O},i}$ for $i=0,1,2,...,w$ be defined as the subnetwork generated by all reactions in $P_i$. Then one of the following holds:
\begin{itemize}
\item [i.] The reaction vectors for $\mathscr{{N}}_{\mathscr{O},i}$ are linearly independent, and the subnetwork $\mathscr{{N}}_{\mathscr{O},i}$ based on $P_i$ forms a forest (i.e., a graph with no cycle) with deficiency 0.
\item [ii.] The reaction vectors are minimally dependent, and the subnetwork $\mathscr{{N}}_{\mathscr{O},i}$ based on $P_i$ forms a forest with deficiency 1.
\item [iii.] The reaction vectors are minimally dependent, and the subnetwork $\mathscr{{N}}_{\mathscr{O},i}$ based on $P_i$ forms a big cycle (with at least three vertices) with deficiency 0.
\end{itemize}
\label{ji254}
\end{lemma}

We denote the subnetwork classes in i, ii, and iii of Lemma \ref{ji254} as Type I, Type II and Type III subnetworks respectively. The classification was extended as follows: an $\mathscr{F}$-subnetwork is of type I, II or III if it contains a $\mathscr{P}$-subnetwork of type I, II or III, respectively. We denote the numbers of subnetworks (i.e., fundamental classes) for Types I, II and III with the symbols $w_I$, $w_{II}$ and $w_{III}$, respectively. We then introduce the following definition \cite{hernandez2}. 

\begin{definition}
An $\mathscr{F}$-decomposition is said to be
\begin{itemize}
\item[i.] {\bf Type I} if it contains Type I subnetwork only.
\item[ii.] {\bf Type II} if it contains Type II subnetwork only.
\item[iii.] {\bf Type III} if it contains Type III subnetwork only.
\end{itemize}
\end{definition}

The following results show that the independence, incidence-independence, and hence, bi-independence of the $\mathscr{F}$-decomposition of a CRN depend on the relationship between the deficiency of the network and the number of Type II subnetworks.

\begin{proposition} \cite{hernandez2}
\label{inde_wII}
If a CRN has independent $\mathscr{F}$-decomposition, then $\delta \le w_{II}$.
\end{proposition}

\begin{corollary}
If the deficiency $\delta > w_{II}$ of a CRN, then the $\mathscr{F}$-decomposition is not independent.
\end{corollary}

\begin{proposition}
\label{inciinde_wII}
If a CRN has incidence-independent $\mathscr{F}$-decomposition, then $\delta \ge w_{II}$.
\end{proposition}

\begin{proof}
Since the CRN has incidence-independent $\mathscr{F}$-decomposition, so $\delta \ge \delta_1 + \delta_2 +...+ \delta_w$. Thus, $\delta \ge \delta_1 + \delta_2 +...+ \delta_w$. Since each Type I or Type III subnetwork has zero deficiency, the deficiency is dependent on the number of Type II subnetworks, i.e., $\delta \ge w_{II}$.
\end{proof}

\begin{corollary}
If the deficiency $\delta < w_{II}$ of a CRN, then the $\mathscr{F}$-decomposition is not incidence-independent.
\end{corollary}

\begin{corollary}
\label{equal_wII}
If a CRN has bi-independent $\mathscr{F}$-decomposition, then $\delta = w_{II}$.
\end{corollary}

\begin{proof}
The conclusion follows directly from Proposition \ref{inde_wII} and Proposition \ref{inciinde_wII}.
\end{proof}

\begin{example}
It was shown in \cite{hernandez2} that the following CRN for $k$-site distributive phosphorylation/ dephosphorylation:
\[\begin{array}{c}
{S_0} + K \mathbin{\lower.3ex\hbox{$\buildrel\textstyle\rightarrow\over
{\smash{\leftarrow}\vphantom{_{\vbox to.5ex{\vss}}}}$}} {S_0}K{\mkern 1mu}  \to {S_1} + K{\mkern 1mu}  \mathbin{\lower.3ex\hbox{$\buildrel\textstyle\rightarrow\over
{\smash{\leftarrow}\vphantom{_{\vbox to.5ex{\vss}}}}$}} {S_1}K \to {S_2} + K{\mkern 1mu}  \mathbin{\lower.3ex\hbox{$\buildrel\textstyle\rightarrow\over
{\smash{\leftarrow}\vphantom{_{\vbox to.5ex{\vss}}}}$}} ... \to {S_k} + K\\
{S_k} + F \mathbin{\lower.3ex\hbox{$\buildrel\textstyle\rightarrow\over
{\smash{\leftarrow}\vphantom{_{\vbox to.5ex{\vss}}}}$}} ...{\mkern 1mu} {\mkern 1mu}  \to {\mkern 1mu} {S_2} + F \mathbin{\lower.3ex\hbox{$\buildrel\textstyle\rightarrow\over
{\smash{\leftarrow}\vphantom{_{\vbox to.5ex{\vss}}}}$}} {S_2}F \to {S_1} + F \mathbin{\lower.3ex\hbox{$\buildrel\textstyle\rightarrow\over
{\smash{\leftarrow}\vphantom{_{\vbox to.5ex{\vss}}}}$}} {S_1}F \to {S_0} + F
\end{array}\]
in \cite{CODH2018} has bi-independent $\mathscr{F}$-decomposition with the following subnetworks:
\[\begin{array}{*{20}{c}}
{{S_i} + K \mathbin{\lower.3ex\hbox{$\buildrel\textstyle\rightarrow\over
{\smash{\leftarrow}\vphantom{_{\vbox to.5ex{\vss}}}}$}} {S_i}K \to {S_{i + 1}} + K}\\
{{S_{i + 1}} + F \mathbin{\lower.3ex\hbox{$\buildrel\textstyle\rightarrow\over
{\smash{\leftarrow}\vphantom{_{\vbox to.5ex{\vss}}}}$}} {S_{i + 1}}F \to {S_i} + F}
\end{array}{\text{ \ \ for }}i = 0,1,...,k - 1.\]
The CRN has $4k+2$ complexes and there are 2 linkage classes. In addition, the rank of the network is $3k$. Hence, the deficiency of the CRN is $\delta = (4k+2)-2-3k=k$. Note that each of the $k$ subnetwork is of Type II (forest of deficiency 1). Thus, the equality in Corollary \ref{equal_wII} is obtained.
\end{example}

We now give an example that determines whether a CRN has the capacity for multistationarity using (1) the MATLAB program (in Appendix \ref{matlabprogram}) that we developed for fundamental decompositions, (2) the DZT, and (3) the FDT.


\begin{example}
We consider the following CRN, endowed with the mass action kinetics.
\[ \begin{array}{lll}
R_1: M_1 \to M_2 & \ \ \ &  R_4: M_2  \to M_4\\
R_2: M_2  \to  M_3  &  \ \ \ & R_5: M_4 \to M_5 \\
R_3: M_3  \to M_1 &  \ \ \ & R_{6}: M_5 \to M_2\\
\end{array}\]
We use the program in Appendix \ref{matlabprogram} with the following input.
{\small
\begin{lstlisting}
model.id = 'SAMPLE'; 
model.name = 'SAMPLE';
model.species = struct('id', {'M1', 'M2', 'M3', 'M4', 'M5'});
model.reaction(1) = struct('id', 'M1->M2', 'reactant', struct('species', {'M1'}, 'stoichiometry', {1}), 'product', struct('species', {'M2'}, 'stoichiometry', {1}), 'reversible', false);
model.reaction(2) = struct('id', 'M2->M3', 'reactant', struct('species', {'M2'}, 'stoichiometry', {1}), 'product', struct('species', {'M3'}, 'stoichiometry', {1}), 'reversible', false);
model.reaction(3) = struct('id', 'M3->M1', 'reactant', struct('species', {'M3'}, 'stoichiometry', {1}), 'product', struct('species', {'M1'}, 'stoichiometry', {1}), 'reversible', false);
model.reaction(4) = struct('id', 'M2->M4', 'reactant', struct('species', {'M2'}, 'stoichiometry', {1}), 'product', struct('species', {'M4'}, 'stoichiometry', {1}), 'reversible', false);
model.reaction(5) = struct('id', 'M4->M5', 'reactant', struct('species', {'M4'}, 'stoichiometry', {1}), 'product', struct('species', {'M5'}, 'stoichiometry', {1}), 'reversible', false);
model.reaction(6) = struct('id', 'M5->M2', 'reactant', struct('species', {'M5'}, 'stoichiometry', {1}), 'product', struct('species', {'M2'}, 'stoichiometry', {1}), 'reversible', false);
inciinde(model)
\end{lstlisting}
}
Then the following output is obtained.
{\small
\begin{lstlisting}
SUBNETWORK 1:
M2->M4
M4->M5
M5->M2
SUBNETWORK 2:
M1->M2
M2->M3
M3->M1
CONCLUSION 1: The F-decomposition is INDEPENDENT.
CONCLUSION 2: The F-decomposition is INCIDENCE-INDEPENDENT.
CONCLUSION 3: The F-decomposition is BI-INDEPENDENT.
\end{lstlisting}
}
The subnetworks are cycles (of deficiency 0). By the DZT, each subnetwork does not have the capacity for multistationarity. It follows from the FDT that the network cannot admit multistationarity.
\end{example}

We now introduce this definition which groups CRNs into two with respect to the fundamental decomposition.

\begin{definition}
An $\mathscr{F}$-decomposition is said to be
\begin{itemize}
\item[i.] {\bf Type Zero} if it contains Type I or Type III subnetwork, and
\item[ii.] {\bf Type One} if it contains at least one Type II subnetwork.
\end{itemize}
\end{definition}

The following proposition generalizes the Deficiency Zero Theorem using the concept of Type Zero independent fundamental decomposition.

\begin{proposition}
Suppose a CRN has Type Zero independent fundamental decomposition. Then
\begin{itemize}
\item[i.] if at least one of the subnetworks is not weakly reversible, then for arbitrary kinetics, the differential equations for the corresponding reaction system cannot admit an equilibrium;
\item[ii.] if at least one of the subnetworks is not weakly reversible, then for arbitrary kinetics, the differential equations for the corresponding reaction system cannot admit a cyclic composition trajectory containing a positive composition; and
\item[iii.] if each subnetwork is weakly reversible, then for any mass action kinetics (but regardless
of the positive values the rate constants take), the resulting differential equations cannot admit multiple equilibria.
\end{itemize}
\end{proposition}

\begin{proof}
This follows from the DZT and the FDT.
\end{proof}

The subnetwork of the Schmitz's carbon cycle model in \cite{fortun2,schmitz} was generalized in \cite{hernandez}, which is an instance of an independent Type III $\mathscr{F}$-decomposition. An illustration of the graph is given in Figure \ref{cyclesgraph1}.

\begin{theorem} \cite{hernandez2}
\label{cycles:graph}
The following family of CRNs has bi-independent Type III $\mathscr{F}$-decomposition such that the $\mathscr{N}_i$'s are precisely the fundamental classes under the decomposition:
$\mathscr{N} =\{\mathscr{N}_i | \mathscr{N}_i = (\mathscr{C}_i, \mathscr{R}_i)\}$ with a (possibly broken) chain of long monomolecular directed cycles, i.e. of length 
$\ge$ 3, and $|\mathscr{C}_i \cap \mathscr{C}_j| \le 1$ if $j = i + 1$ for $i = 0,1,...,k-1$.
\end{theorem}

\begin{figure*}
\begin{center}
\includegraphics[width=14cm,height=28cm,keepaspectratio]{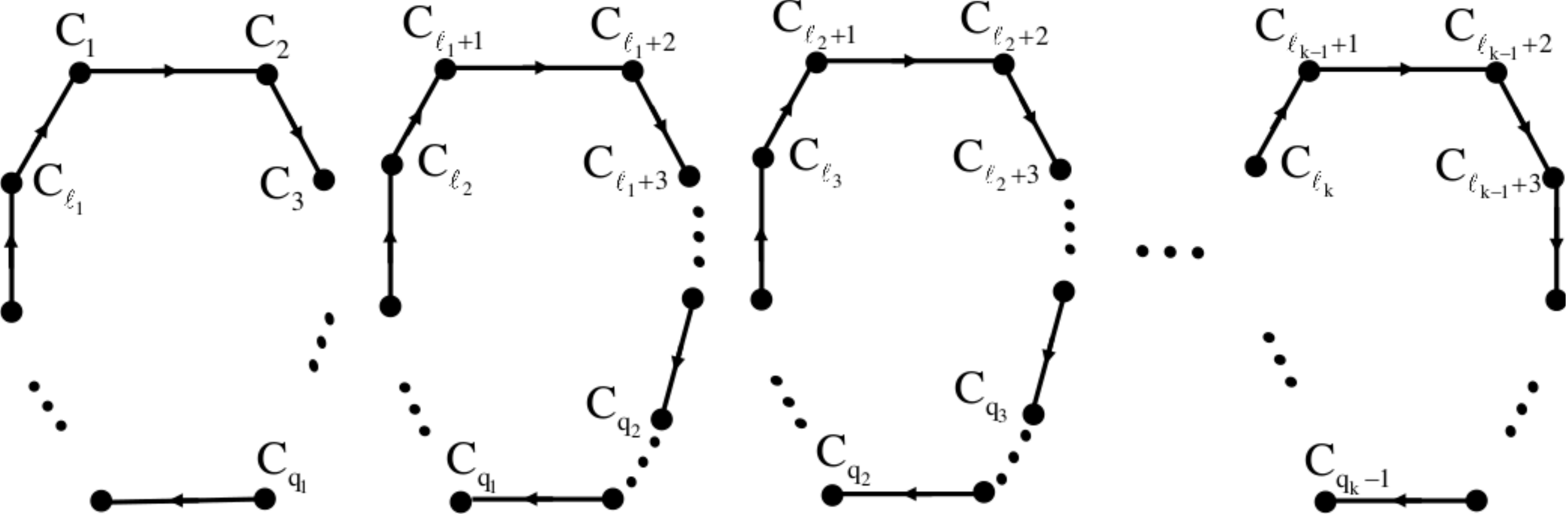}
\caption{An illustration of the graph with no break in Theorem \ref{cycles:graph}  \cite{hernandez2}.}
\label{cyclesgraph1}
\end{center}
\end{figure*}

The following result analyzes whether the mass action kinetics with underlying CRN in Theorem \ref{cycles:graph} has the capacity for multistationarity.

\begin{example}
\label{cycles:graph2}
The CRN
$\mathscr{N} =\{\mathscr{N}_i | \mathscr{N}_i = (\mathscr{C}_i, \mathscr{R}_i)\}$ with a (possibly broken) chain of long monomolecular directed cycles, i.e. of length 
$\ge$ 3, and $|\mathscr{C}_i \cap \mathscr{C}_j| \le 1$ if $j = i + 1$ for $i = 0,1,...,k-1$, endowed with mass action kinetics, does not have the capacity for multistationarity, i.e., the system admits at most one equilibrium.
\end{example}

\section{Fundamental Decompositions of CRNs under the CF-RM Transformation}
\label{decomposition_CFRM}

In this section, we present a transformation method whose key property is that it maps an irreversible reaction (a reversible pair of reactions) of the original system to an irreversible reaction (a reversible pair of reactions) of the target system. In other words, it is reversibility and irreversibility (RI) preserving.
This method was based on the generic CF-RM method (transformation of complex factorizable kinetics by reactant multiples)
which converts a PL-NDK to a PL-RDK system.
We add in the notation CF-RI  a sub-index ``+'' for two reasons: to indicate the ``positive'' (or preserving) relation and to highlight its partial coincidence with the CF-RM$_+$ variant of CF-RM. The following presents the steps of the CF-RM transformation method in \cite{cfrm}. 
\begin{itemize}
\item [1.] Determine the set of reactant complexes $\rho \left(\mathscr{R} \right)$.
\item [2.] Leave each CF-reactant complex unchanged.
\item [3.] At an NF-reactant complex, select a CF-subset containing the highest number of reactions and leave this CF-subset unchanged. For each of the remaining $N_R(y)-1$ CF-subsets, choose successively a multiple of $y$ which is not among the
current set of reactants. Different procedures are possible for the selection of a new reactant as long as
it is different from those in the current reactant set. After each choice, the current set is updated.
\end{itemize}

CF-RM$_+$ is a variant of CF-RM. All the steps are identical with the generic CF-RM method except that it uses additional criteria in the selection of the new reactant multiples. CF-RM$_+$ chooses the reactant multiple so that the new reactant differs from all existing complexes and all the new product complexes in the CF-subset also differ from all existing complexes \cite{cfrm}. Note that the CF-RM$_+$ method given in \cite{cfrm} updates the set of current complexes and complexes in the transform after each CF-subset of an NF-node is processed. The CF-RI$_+$ method proceeds as follows:
\begin{itemize}
\item[1.] Determine the reactant set $\rho({\mathscr{R})}$ and identify the subset $\rho({\mathscr{R})}_{CF}$ of CF-nodes.
\item[2.] If the reaction set ${\mathscr{R}_y:=\rho^{-1}(y)}$
of a CF-node $y$ has no reversible reaction with an NF-node, then it is left unchanged.  
\item[3.] At an NF-node without reversible reactions, carry out the steps of CF-RM$_+$.  
\item[4.] At an NF-node with a reversible reaction, among the CF-subsets without a reversible reaction (if there are any), select one with the highest number of reactions and leave this unchanged.
\item[5.] For the remaining CF-subsets without a reversible reaction, carry out CF-RM$_+$.
\item[6.] For a CF-subset with a reversible reaction, carry out CF-RM$_+$, but in addition, for each reversible reaction, also for the CF-subset of the reverse reaction (with the same ``catalytic'' complex). If the reactant complex of the reverse reaction is an NF-node, this additional step removes the original CF-subset from the reaction set of that NF-node. If this removal transforms the NF-node to a CF-node, then remove the node from the list of NF-nodes.
\end{itemize}

The following theorem relates the independence of the fundamental decomposition of a system and its CF-RI$_+$ transform.

\begin{theorem} \cite{hernandez2}
\label{independent2}
Let $({\mathscr{N}}, K)$ be a PL-NDK system and $({\mathscr{N}}_{RI}, K_{RI})$ a CF-RI$_+$ transform. Then
\begin{itemize}
\item[i.] for any orientation ${\mathscr{O}}$ of ${\mathscr{N}}$, $|{\mathscr{O}}| = |{\mathscr{O}_{RI}}|$, and
\item[ii.] the ${\mathscr{F}}$-decomposition of ${\mathscr{N}}$ is independent if and only if the ${\mathscr{F}}$-decomposition of ${\mathscr{N}}_{RI}$ is independent.
\end{itemize}
\end{theorem}

The following theorem is a restatement of Theorems \ref{PifC} and \ref{independent2}. This implies that with or without the application of the CF-RM Transformation, the computation in the Higher Deficiency Algorithm are the same with the assumption that the $\mathscr{P}$-decomposition or the $\mathscr{F}$-decomposition is independent. We assume that if we have the symbol $^*$ in the notation, we are dealing with the CF-RM applied to it.

\begin{theorem}
\label{independent}
Let $\left(\mathscr{S},\mathscr{C},\mathscr{R}\right)$ be a CRN and $\mathscr{O}$ be an orientation. Suppose under the CF-RM transformation, the reversibility and irreversibility of the reactions are retained, with an induced orientation $\mathscr{O^*}$. Then $|\mathscr{O}|=|\mathscr{O}^*|$ and the following are equivalent.
\begin{itemize}
\item[i.] The $\mathscr{P}$-decomposition is independent.
\item[ii.] The $\mathscr{F}$-decomposition is independent.
\item[iii.] The $\mathscr{P}^*$-decomposition is independent.
\item[iv.] The $\mathscr{F}^*$-decomposition is independent.
\end{itemize}
\end{theorem}

\begin{proof}
(i) $ \Leftrightarrow $ (iii) follows from Theorem \ref{independent2}.
(i) $ \Leftrightarrow $ (ii) follows from Theorem \ref{PifC}.
(iii) $ \Leftrightarrow $ (iv) follows from Theorem \ref{PifC}.
Finally, (ii) $ \Leftrightarrow $ (iv) follows from Theorem \ref{independent2}.
\end{proof}

\begin{remark}
If reversibility and irreversibility of the reactions under the CF-RM are retained, Theorem \ref{independent} implies that the computation in the HDA are just the same for independent $\mathscr{P}$-decomposition or $\mathscr{F}$-decomposition.
\end{remark}

The following example will show that we do not have equivalence statements as in Theorem \ref{independent} for incidence-independence assuming the same conditions.

\begin{example}
Consider the following reaction network given in \cite{hernandez}.\\
$R_1: 0 \to A_1$\\
$R_2: A_1 \to 0$\\
$R_3: A_1 \to 2A_1$\\
$R_4: 2A_1 \to 0$

The network has a unique linkage class which is precisely a terminal strong linkage class. The rank of the network is 1 and has a deficiency of 1. Moreover, it is weakly reversible. Consider the following kinetic order values: for $R_1: 0$,  for $R_2: 0.5$, for $R_3: 1$, and for $R_4: 0.5$. Thus, the system is PL-NDK. Using the CF-RM transformation, we modify $R_3: 3A_1 \to 4A_1$. We obtain the following network with deficiency 2 \cite{hernandez}.

Note that the reversibility and the irreversibility of the reactions remain the same after the application of CF-RM. We can actually verify that indeed without using the CF-RM, one can directly apply the HDA and the computation yields the same results.
\end{example}

\begin{example}
\label{not_incidence_independent}
Consider the following reaction network with its kinetic order matrix.
$$\bordermatrix{%
& A & B & C \cr
R_1: A \to B & 1 & 0 & 0  \cr
R_2: B \to C & 0 & 1 & 0  \cr
R_3: A \to 0 & 0.5 & 0 & 0  \cr
R_4: B \to 0 & 0 & 0.5 & 0  \cr
R_5: C \to 0 & 0 & 0 & 1  \cr
}$$
The following is a basis for $Ker L_{\mathscr{O}}$.
\[\left( {\begin{array}{*{20}{c}}
1&{ 1}\\
1&0\\
-1&-1\\
0&1\\
1&0
\end{array}} \right)\]
Consequently, the equivalence classes coincide with the fundamental classes.
$P_1=C_1=\{R_1,R_3\}$, $P_2=C_2=\{R_2,R_5\}$, $P_3=C_3=\{R_4\}$

Below are the incidence matrices of the equivalence classes (which are equal to the corresponding fundamental classes).
$$\bordermatrix{%
& R_1 & R_3 \cr
A & -1 & 0   \cr
B & 1 & -1  \cr
0 & 0 & 1  \cr
} \ \ \ 
\bordermatrix{%
& R_2 & R_5 \cr
B & -1 & 0   \cr
C & 1 & -1  \cr
0 & 0 & 1  \cr
} \ \ \ 
\bordermatrix{%
& R_4 \cr
B & -1   \cr
0 & 1  \cr
}$$
Hence, the direct sum of the incidence matrices has rank 5.
On the other hand, below is the incidence matrix of the whole network.
$$\bordermatrix{%
& R_1 & R_2 & R_3 & R_4 & R_5 \cr
A  & -1 & 0 & -1 & 0 & 0 \cr
B  & 1 & -1 & 0 & -1 & 0 \cr
C  & 0 & 1 & 0 & 0 & -1 \cr
0  & 0 & 0 & 1 & 1 & 1 \cr
}$$
We see that the incidence matrix has rank less than 5.
Therefore, the ${\mathscr{P}}$-decomposition is not incidence-independent. Also, the ${\mathscr{F}}$-decomposition is not incidence-independent.

Now, we use the CF-RM to transform the PL-NDK to a dynamically equivalent PL-RDK.
The branching nodes are $A$ and $B$. For the node $A$, the branching reactions are $R_1$ and $R_3$. We choose to change $R_3: A \to 0$ to $R_3^*: 2A \to A$. For the node $B$, the branching reactions are $R_2$ and $R_4$. We choose to change $R_4: B \to 0$ to $R_4^*: 2B \to B$. Hence, we have the following resulting reaction network.
$$\bordermatrix{%
& A & B & C \cr
R_1: A \to B & 1 & 0 & 0  \cr
R_2: B \to C & 0 & 1 & 0  \cr
R_3^*: 2A \to A & 0.5 & 0 & 0  \cr
R_4^*: 2B \to B & 0 & 0.5 & 0  \cr
R_5: C \to 0 & 0 & 0 & 1  \cr
}$$

Below are the incidence matrices of the equivalence classes (which are equal to the corresponding fundamental classes).
$$\bordermatrix{%
& R_1 & R_3^* \cr
A & -1 & 1   \cr
B & 1 & 0  \cr
2A & 0 & -1  \cr
} \ \ \ 
\bordermatrix{%
& R_2 & R_5 \cr
B & -1 & 0   \cr
C & 1 & -1  \cr
0 & 0 & 1  \cr
} \ \ \ 
\bordermatrix{%
& R_4^* \cr
2B & -1   \cr
B & 1  \cr
}$$
Hence, the direct sum of the incidence matrices has rank 5.
On the other hand, below is the incidence matrix of the whole network.
$$\bordermatrix{%
& R_1 & R_2 & R_3^* & R_4^* & R_5 \cr
A  & -1 & 0 & 1 & 0 & 0 \cr
B  & 1 & -1 & 0 & 1 & 0 \cr
C  & 0 & 1 & 0 & 0 & -1 \cr
2A  & 0 & 0 & -1 & 0 & 0 \cr
2B  & 0 & 0 & 0 & -1 & 0 \cr
0  & 0 & 0 & 0 & 0 & 1 \cr
}$$
Thus, the incidence matrix has rank 5.
Therefore, the ${\mathscr{P}}$-decomposition is incidence-independent. Also, the ${\mathscr{F}}$-decomposition is incidence-independent.
\end{example}

\begin{remark}
Even with the restriction that the reversibility and irreversibility of the reactions under the CF-RM are retained, Example \ref{not_incidence_independent} shows that if $\mathscr{P}$-decomposition is incidence-independent, it does not follow that the $\mathscr{P}^*$-decomposition is also incidence-independent. Likewise, the incidence-independence of the $\mathscr{F}$-decomposition does not imply the incidence-independence of the $\mathscr{F}^*$-decomposition.
\end{remark}

\section{Conclusions and Outlook}
\label{sec:conclusion}
We summarize our results and provide some direction for future research.

\begin{itemize}
\item[1.] We illustrated the ${\mathscr{O}}$-, ${\mathscr{P}}$-, and ${\mathscr{F}}$-decompositions underlying the higher deficiency algorithm for mass action kinetics and the multistationarity algorithm for power-law kinetics. We also derived properties of these decompositions.
\item[2.] We employed the use of (1) the program that we created to determine the fundamental classes and whether or not the decomposition is independent and incidence-independent, (2) the DZT, and (3) the FDT, to provide a simple solution of multistationarity of the CRN of the generalization of a subnetwork of Schmitz's carbon cycle model by Hernandez et al. endowed with mass action kinetics.
\item[3.] We stated equivalent statements regarding the $\mathscr{P}$-, the $\mathscr{P}^*$-, the $\mathscr{F}$-, and the $\mathscr{F}^*$-decompositions with the assumption that the reversibility and irreversibility of the reactions under the CF-RM are retained. In this case, $|\mathscr{O}|=|\mathscr{O}^*|$.
\item[4.] Even with the restriction that the reversibility and irreversibility of the reactions under the CF-RM are preserved, Example \ref{not_incidence_independent} shows that if $\mathscr{F}$-decomposition is incidence-independent, it does not follow that the $\mathscr{F}^*$-decomposition is also incidence-independent. We provided a counterexample for this case.
\item[5.] One may look into necessary conditions to establish equivalence statements for fundamental incidence-independent decompositions. In addition, Type One $\mathscr{F}$-decompositions can also be considered for further study.
\end{itemize}

\appendix
\section{Nomenclature}
\label{nomenclature:appendix}
\subsection{List of abbreviations}
\begin{tabular}{ll}
\noalign{\smallskip}\hline\noalign{\smallskip}
Abbreviation& Meaning \\
\noalign{\smallskip}\hline\noalign{\smallskip}
CF& complex factorizable \\
CKS& chemical kinetic system\\
CRN& chemical reaction network\\
HDA& higher deficiency algorithm\\
MAK& mass action kinetics\\
MSA& multistationarity algorithm\\
PLK& power-law kinetics\\
PL-NDK& power-law non-reactant-determined kinetics\\
PL-RDK& power-law reactant-determined kinetics\\
SFRF& species formation rate function\\
\noalign{\smallskip}\hline
\end{tabular}
\subsection{List of important symbols}
\begin{tabular}{ll}
\noalign{\smallskip}\hline\noalign{\smallskip}
Meaning& Symbol \\
\noalign{\smallskip}\hline\noalign{\smallskip}
deficiency& $\delta$  \\
dimension of the stoichiometric subspace& $s$   \\
incidence map& $I_a$\\
molecularity matrix& $Y$\\
number of complexes& $n$\\
number of linkage classes& $l$\\
number of strong linkage classes& $sl$\\
orientation& $\mathscr{O}$\\
stoichiometric matrix& $N$\\
stoichiometric subspace& $S$\\
subnetwork of $\mathscr{N}$ with respect to $\mathscr{O}$& $\mathscr{N}_\mathscr{O}$\\
\noalign{\smallskip}\hline
\end{tabular}

\section{A MATLAB program}
\label{matlabprogram}
We now provide a MATLAB program that computes for the subnetworks under the fundamental decomposition of a reaction network. It also determines whether the decomposition is independent, incidence-independent and bi-independent. We use the preliminary steps of the program of Soranzo and Altafini \cite{soranzo} to come up with our own program. We should install the free software ERNEST in our MATLAB environment. The script was named inciinde.m.
{\scriptsize
\begin{lstlisting}
function [ret] = inciinde(model)
species = {model.species.id};
n = numel(species); % number of species
reactions = {model.reaction.id};
reactant_complexes = []; % matrix of reactant complexes (species x irrev. reactions)
product_complexes = []; % matrix of product complexes (species x irrev. reactions)
rr = numel(reactions); % number of reactions (counting reversible reactions as one)
S = []; % stoichiometric matrix (species x irrev. reactions)
SnoRev = []; % stoichiometric matrix (without reverse of revesible reaction)
P_i = []; %equivalence classes
P_i_new = []; %equivalence classes
StoichMatrixForm = []; 
sr_edges = cell(0, 4);
arr=[];
for i = 1:numel(reactions)
    if isfield(model.reaction(i), 'modifier') && ~isempty(model.reaction(i).modifier)
        warning(['Reaction ' num2str(i) ' contains modifiers, which will be ignored. Specify all species in a reaction as reactants or products.'])
    end
    reactant_complexes(:, end+1) = zeros(n, 1);
    for j = 1:numel(model.reaction(i).reactant)
        reactant_complexes(find(strcmp(model.reaction(i).reactant(j).species, species), 1), end) = model.reaction(i).reactant(j).stoichiometry;
    end
    product_complexes(:, end+1) = zeros(n, 1);
    for j = 1:numel(model.reaction(i).product)
        product_complexes(find(strcmp(model.reaction(i).product(j).species, species), 1), end) = model.reaction(i).product(j).stoichiometry;
    end
    SnoRev(:, end + 1) = product_complexes(:, end) - reactant_complexes(:, end);  
end
clear label
for i = 1:numel(reactions)
    if isfield(model.reaction(i), 'modifier') && ~isempty(model.reaction(i).modifier)
        warning(['Reaction ' num2str(i) ' contains modifiers, which will be ignored. Specify all species in a reaction as reactants or products.'])
    end
    reactant_complexes(:, end+1) = zeros(n, 1);
    for j = 1:numel(model.reaction(i).reactant)
        reactant_complexes(find(strcmp(model.reaction(i).reactant(j).species, species), 1), end) = model.reaction(i).reactant(j).stoichiometry;
    end
    product_complexes(:, end+1) = zeros(n, 1);
    for j = 1:numel(model.reaction(i).product)
        product_complexes(find(strcmp(model.reaction(i).product(j).species, species), 1), end) = model.reaction(i).product(j).stoichiometry;
    end
    S(:, end + 1) = product_complexes(:, end) - reactant_complexes(:, end);
    if model.reaction(i).reversible
        reactant_complexes(:, end+1) = product_complexes(:, end);
        product_complexes(:, end+1) = reactant_complexes(:, end-1);
        S(:, end + 1) = -S(:, end);
    end
    if numel(model.reaction(i).reactant) > 0 && numel(model.reaction(i).product) > 0
        label = [num2str(model.reaction(i).reactant(1).stoichiometry) ' ' model.reaction(i).reactant(1).species];
        for j = 2:numel(model.reaction(i).reactant)
            label = [label ' + ' num2str(model.reaction(i).reactant(j).stoichiometry) ' ' model.reaction(i).reactant(j).species];
        end
        for j = 1:numel(model.reaction(i).reactant)
            sr_edges(end+1, :) = {model.reaction(i).reactant(j).species, reactions{i}, label, model.reaction(i).reactant(j).stoichiometry};
        end
        label = [num2str(model.reaction(i).product(1).stoichiometry) ' ' model.reaction(i).product(1).species];
        for j = 2:numel(model.reaction(i).product)
            label = [label ' + ' num2str(model.reaction(i).product(j).stoichiometry) ' ' model.reaction(i).product(j).species];
        end
        for j = 1:numel(model.reaction(i).product)
            sr_edges(end+1, :) = {model.reaction(i).product(j).species, reactions{i}, label, model.reaction(i).product(j).stoichiometry};
        end
    end
end
clear label
[Y, ind, ind2] = unique([reactant_complexes product_complexes]', 'rows'); % ind2(i) is the index in Y of the reactant complex in reaction i, ind(i + r) is the index in Y of the product complex in reaction i
Y = Y'; % complexes matrix (species x complexes)
m = size(Y, 2); % number of complexes
reacts_to = false(m, m); % matrix (complexes x complexes) for the reacts_to relation: reacts_to(i, j) = true iff i->j
r = size(reactant_complexes, 2); % number of irrev. reactions
reacts_in = zeros(m, r); % matrix (complexes x irrev. reactions) for the reacts_in relation: (reacts_in(i, r) = -1 && reacts_in(j, r) = 1) iff i->j
for i = 1:r
    reacts_to(ind2(i), ind2(i + r)) = true;
    reacts_in(ind2(i), i) = -1;
    reacts_in(ind2(i+r), i) = 1; %incidence
end
is_reversible = isequal(reacts_to, reacts_to'); %test for reversibility
complexes_ugraph_cc = connected_components(umultigraph(reacts_to | reacts_to')); % linkage classes
l = max(complexes_ugraph_cc); % number of linkage classes
if is_reversible
    complexes_graph_scc = complexes_ugraph_cc;
else
    complexes_graph_scc = strongly_connected_components(multigraph(reacts_to)); % strong-linkage classes
end
n_slc = max(complexes_graph_scc); % number of strong-linkage classes
is_weakly_reversible = n_slc == l; % the reaction network is weakly reversible if and only if each linkage class is a strong-linkage class
s = rank(S); % reaction network rank
d = m - l - s; % reaction network deficiency

%computing Ker L_O
Nsp = null(SnoRev,'r');
[NspRow,NspCol]=size(Nsp);
NspEqual=Nsp;
for i=1:NspRow
    if NspEqual(i,:)~=zeros(NspCol,1)
    NspEqual(i,:)=NspEqual(i,:)./gcd(sym(NspEqual(i,:)));
    end
end
for i=1:NspRow
    for j=1:NspRow
        if NspEqual(i,:)==-NspEqual(j,:);
        NspEqual(j,:)=NspEqual(i,:);
        end
    end
end
NspEqual;
NspUnique=unique(Nsp, 'rows', 'stable');
TranposeNspUnique=NspUnique.';
size(NspUnique);
NspOrtho = null(TranposeNspUnique,'r');
NspOrthoTranspose = NspOrtho.';
size(NspOrthoTranspose);
SimplifiedNspOrthoTranspose = rref(NspOrthoTranspose);
TansposeSimplifiedNspOrthoTranspose = SimplifiedNspOrthoTranspose.';
TansposeSimplifiedNspOrthoTranspose2 = TansposeSimplifiedNspOrthoTranspose.';
size(TranposeNspUnique);
size(TansposeSimplifiedNspOrthoTranspose);
TranposeNspUnique*TansposeSimplifiedNspOrthoTranspose;        
UniqueNspMatrix = unique(NspEqual, 'rows'); %take the unique NSP complexes
[UNSPMsizeRow,UNSPMsizeCol] = size(UniqueNspMatrix); 
[NspSizeRow,NspSizeCol] = size(Nsp); 
%The following command is to identify the location of reactions (equivalence classes).
[LiaRR,LocbRR] = ismember(Nsp, UniqueNspMatrix, 'rows');
IdentifyLocbRR = LocbRR;
IdentifyUniqLocbRR = unique(IdentifyLocbRR(~isnan(IdentifyLocbRR)));
histIdentifyLocbRR=histc(IdentifyLocbRR,IdentifyUniqLocbRR);
UniqueRR = IdentifyUniqLocbRR(histIdentifyLocbRR >1);
[RRsizeRow,RRsizeCol] = size(UniqueRR);
%The following command is to identify the location of reactions (equivalence classes).
[Lia,Locb] = ismember(NspEqual, UniqueNspMatrix, 'rows');
IdentifyLocb = Locb;
IdentifyUniqLocb = unique(IdentifyLocb(~isnan(IdentifyLocb)));
histIdentifyLocb=histc(IdentifyLocb,IdentifyUniqLocb);
UniqueReactionRow = IdentifyUniqLocb(histIdentifyLocb >=1);
[URsizeRow,URsizeCol] = size(UniqueReactionRow);
UniqueReactionRowMatrixTik = zeros(UNSPMsizeRow,numel(reactions));
for k=1:URsizeRow
    recordALL=find(ismember(IdentifyLocb,UniqueReactionRow(k)));
    [Res,LocRes] = ismember(Nsp,(zeros(NspSizeCol,1).'),'rows');
    not(Res);
    if not(Res) & not(isempty(Nsp))
    record=find(ismember(IdentifyLocb,UniqueReactionRow(k)));
    [URsizeRecordRow,URsizeRecordCol] = size(record);
    for i=1:URsizeRecordRow
        %disp(model.reaction(record(i)).id);
        %disp(model.reaction(record(i)).reversible);
    end
    else
    Res;
    [Res,LocRes] = ismember(Nsp,(zeros(NspSizeCol,1).'),'rows');
    if isempty(Nsp) | Res
    %fprintf('The zeroth fundamental class F0 has the following reaction(s).');
    record0=find(ismember(IdentifyLocb,UniqueReactionRow(k)));
    %fprintf('\n');
    [URsizeRecordRow0,URsizeRecordCol0] = size(record0);
    for i=1:URsizeRecordRow0
        %disp(model.reaction(record0(i)).id);
        %disp(model.reaction(record0(i)).reversible);
    end
    else 
     
    %fprintf('The fundamental class F%d has the following reaction(s).', UniqueReactionRow(k)-1);
    record=find(ismember(IdentifyLocb,UniqueReactionRow(k)));
    %fprintf('\n');
    [URsizeRecordRow,URsizeRecordCol] = size(record);
    for i=1:URsizeRecordRow
        %disp(model.reaction(record(i)).id);
        %disp(model.reaction(record(i)).reversible); 
    end   
    end
    end        
end 
SN=Locb.';
Locb=[];
subnetworks = {SN};
subnetworks2 = [SN];
subnetworks3 = unique(subnetworks2);
subn = numel(subnetworks3);
reactant_complexes = []; % matrix of reactant complexes (species x irrev. reactions)
product_complexes = []; % matrix of product complexes (species x irrev. reactions)
S = []; % stoichiometric matrix (species x irrev. reactions)
sum = 0;
sum2 = 0;
StoichMatrixForm = []; 
sr_edges = cell(0, 4);
%NOTE: If you want the reverse reaction of a reversible reaction to be included in a different subnetwork,
%use 2 model reactions and treat the two reactions as irreversible.

[Lia,Locb] = ismember(subnetworks2, subnetworks3);
IdentifyLocb = Locb;
IdentifyUniqLocb = unique(IdentifyLocb(~isnan(IdentifyLocb)));
histIdentifyLocb=histc(IdentifyLocb,IdentifyUniqLocb);
UniqueReactionRow = IdentifyUniqLocb(histIdentifyLocb >=1);
[URsizeRow,URsizeCol] = size((UniqueReactionRow).');

for k=1:URsizeRow
    record=find(ismember(IdentifyLocb,UniqueReactionRow(k)));
    [URsizeRecordRow,URsizeRecordCol] = size(record);
    fprintf('SUBNETWORK %d:', UniqueReactionRow(k))
    fprintf('\n')
    for i=1:URsizeRecordCol
        disp(model.reaction(record(i)).id);

    if isfield(model.reaction(record(i)), 'modifier') && ~isempty(model.reaction(record(i)).modifier)
        warning(['Reaction ' num2str(record(i)) ' contains modifiers, which will be ignored. Specify all species in a reaction as reactants or products.'])
    end
    reactant_complexes(:, end+1) = zeros(n, 1);
    for j = 1:numel(model.reaction(record(i)).reactant)
        reactant_complexes(find(strcmp(model.reaction(record(i)).reactant(j).species, species), 1), end) = model.reaction(record(i)).reactant(j).stoichiometry;
    end
    product_complexes(:, end+1) = zeros(n, 1);
    for j = 1:numel(model.reaction(record(i)).product)
        product_complexes(find(strcmp(model.reaction(record(i)).product(j).species, species), 1), end) = model.reaction(record(i)).product(j).stoichiometry;
    end
    S(:, end + 1) = product_complexes(:, end) - reactant_complexes(:, end);
    if model.reaction(record(i)).reversible
        reactant_complexes(:, end+1) = product_complexes(:, end);
        product_complexes(:, end+1) = reactant_complexes(:, end-1);
        S(:, end + 1) = -S(:, end);
    end
    if numel(model.reaction(record(i)).reactant) > 0 && numel(model.reaction(record(i)).product) > 0
        label = [num2str(model.reaction(record(i)).reactant(1).stoichiometry) ' ' model.reaction(record(i)).reactant(1).species];
        for j = 2:numel(model.reaction(record(i)).reactant)
            label = [label ' + ' num2str(model.reaction(record(i)).reactant(j).stoichiometry) ' ' model.reaction(record(i)).reactant(j).species];
        end
        for j = 1:numel(model.reaction(record(i)).reactant)
            sr_edges(end+1, :) = {model.reaction(record(i)).reactant(j).species, reactions{record(i)}, label, model.reaction(record(i)).reactant(j).stoichiometry};
        end
        label = [num2str(model.reaction(record(i)).product(1).stoichiometry) ' ' model.reaction(record(i)).product(1).species];
        for j = 2:numel(model.reaction(record(i)).product)
            label = [label ' + ' num2str(model.reaction(record(i)).product(j).stoichiometry) ' ' model.reaction(record(i)).product(j).species];
        end
        for j = 1:numel(model.reaction(record(i)).product)
            sr_edges(end+1, :) = {model.reaction(record(i)).product(j).species, reactions{record(i)}, label, model.reaction(record(i)).product(j).stoichiometry};
        end
end
clear label
[Y, ind, ind2] = unique([reactant_complexes product_complexes]', 'rows'); % ind2(i) is the index in Y of the reactant complex in reaction i, ind(i + r) is the index in Y of the product complex in reaction i
Y = Y'; % complexes matrix (species x complexes)
m = size(Y, 2); % number of complexes
reacts_to = false(m, m); % matrix (complexes x complexes) for the reacts_to relation: reacts_to(i, j) = true iff i->j
r = size(reactant_complexes, 2); % number of irrev. reactions
reacts_in = zeros(m, r); % matrix (complexes x irrev. reactions) for the reacts_in relation: (reacts_in(i, r) = -1 && reacts_in(j, r) = 1) iff i->j
for i = 1:r
    reacts_to(ind2(i), ind2(i + r)) = true;
    reacts_in(ind2(i), i) = -1;
    reacts_in(ind2(i + r), i) = 1;%incidence
end
is_reversible = isequal(reacts_to, reacts_to'); %test for reversibility
complexes_ugraph_cc = connected_components(umultigraph(reacts_to | reacts_to')); % linkage classes
l = max(complexes_ugraph_cc); % number of linkage classes
if is_reversible
    complexes_graph_scc = complexes_ugraph_cc;
else
    complexes_graph_scc = strongly_connected_components(multigraph(reacts_to)); % strong-linkage classes
end
n_slc = max(complexes_graph_scc); % number of strong-linkage classes
is_weakly_reversible = n_slc == l; % the reaction network is weakly reversible if and only if each linkage class is a strong-linkage class
s = rank(S); % reaction network rank
rr = numel(reactions); % number of reactions (counting reversible reactions as one)
end
%fprintf('The stoichiometric subspace of SUBNETWORK %d is:', UniqueReactionRow(k))
S;
%fprintf('The rank of SUBNETWORK %d is:', UniqueReactionRow(k))
s = rank(S);
m;
l;
n_slc;
%fprintf('The value of (n-l) for the SUBNETWORK %d is:', UniqueReactionRow(k))
diff=m-l;
%fprintf('The value of (n-l) for the SUBNETWORK %d is:', UniqueReactionRow(k))
arr=[arr; UniqueReactionRow(k) s diff];
S = [];
reactant_complexes = [];
product_complexes = [];
sum = sum + s;
sum2 = sum2 + diff;
end

fprintf('Summary of the the values of s (2ND COL) and n-l (3RD COL) of SUBNETWORK i:')
summary=[];
summary=arr;
summary
fprintf('The SUM of the RANKS of the SUBNETWORKS is:')
fprintf('\n')
sum
fprintf('The SUM of the values of of (n-l) of the SUBNETWORKS is:')
fprintf('\n')
sum2
S=[];
for i = 1:numel(reactions)
    if isfield(model.reaction(i), 'modifier') && ~isempty(model.reaction(i).modifier)
        warning(['Reaction ' num2str(i) ' contains modifiers, which will be ignored. Specify all species in a reaction as reactants or products.'])
    end
    reactant_complexes(:, end+1) = zeros(n, 1);
    for j = 1:numel(model.reaction(i).reactant)
        reactant_complexes(find(strcmp(model.reaction(i).reactant(j).species, species), 1), end) = model.reaction(i).reactant(j).stoichiometry;
    end
    product_complexes(:, end+1) = zeros(n, 1);
    for j = 1:numel(model.reaction(i).product)
        product_complexes(find(strcmp(model.reaction(i).product(j).species, species), 1), end) = model.reaction(i).product(j).stoichiometry;
    end  
    S(:, end + 1) = product_complexes(:, end) - reactant_complexes(:, end);
    if model.reaction(i).reversible
        reactant_complexes(:, end+1) = product_complexes(:, end);
        product_complexes(:, end+1) = reactant_complexes(:, end-1);
        S(:, end + 1) = -S(:, end);
    end
    if numel(model.reaction(i).reactant) > 0 && numel(model.reaction(i).product) > 0
        label = [num2str(model.reaction(i).reactant(1).stoichiometry) ' ' model.reaction(i).reactant(1).species];
        for j = 2:numel(model.reaction(i).reactant)
            label = [label ' + ' num2str(model.reaction(i).reactant(j).stoichiometry) ' ' model.reaction(i).reactant(j).species];
        end
        for j = 1:numel(model.reaction(i).reactant)
            sr_edges(end+1, :) = {model.reaction(i).reactant(j).species, reactions{i}, label, model.reaction(i).reactant(j).stoichiometry};
        end
        label = [num2str(model.reaction(i).product(1).stoichiometry) ' ' model.reaction(i).product(1).species];
        for j = 2:numel(model.reaction(i).product)
            label = [label ' + ' num2str(model.reaction(i).product(j).stoichiometry) ' ' model.reaction(i).product(j).species];
        end
        for j = 1:numel(model.reaction(i).product)
            sr_edges(end+1, :) = {model.reaction(i).product(j).species, reactions{i}, label, model.reaction(i).product(j).stoichiometry};
        end
    end
end
%fprintf('The stoichiometric subspace of the WHOLE NETWORK is:')
S;
s=rank(S);
fprintf('The rank of the WHOLE NETWORK is:')
s
fprintf('\n')

clear label
[Y, ind, ind2] = unique([reactant_complexes product_complexes]', 'rows'); % ind2(i) is the index in Y of the reactant complex in reaction i, ind(i + r) is the index in Y of the product complex in reaction i
Y = Y'; % complexes matrix (species x complexes)
m = size(Y, 2); % number of complexes
reacts_to = false(m, m); % matrix (complexes x complexes) for the reacts_to relation: reacts_to(i, j) = true iff i->j
r = size(reactant_complexes, 2); % number of irrev. reactions
reacts_in = zeros(m, r); % matrix (complexes x irrev. reactions) for the reacts_in relation: (reacts_in(i, r) = -1 && reacts_in(j, r) = 1) iff i->j
for i = 1:r
    reacts_to(ind2(i), ind2(i + r)) = true;
    reacts_in(ind2(i), i) = -1;
    reacts_in(ind2(i+r), i) = 1;%incidence
end
is_reversible = isequal(reacts_to, reacts_to'); %test for reversibility
complexes_ugraph_cc = connected_components(umultigraph(reacts_to | reacts_to')); % linkage classes
l = max(complexes_ugraph_cc); % number of linkage classes
if is_reversible
    complexes_graph_scc = complexes_ugraph_cc;
else
    complexes_graph_scc = strongly_connected_components(multigraph(reacts_to)); % strong-linkage classes
end
n_slc = max(complexes_graph_scc); % number of strong-linkage classes
is_weakly_reversible = n_slc == l; % the reaction network is weakly reversible if and only if each linkage class is a strong-linkage class
s = rank(S); % reaction network rank
fprintf('The value of (n-l) for the WHOLE NETWORK is:')
diff=m-l
fprintf('NOTE: The subnetworks given above correspond to the fundamental class under the F-decomposition.')
fprintf('\n')
if sum==s
   fprintf('CONCLUSION 1: The F-decomposition is INDEPENDENT.')
else
   fprintf('CONCLUSION 1: The F-decomposition is NOT INDEPENDENT.')
end
fprintf('\n')
if sum2==diff
   fprintf('CONCLUSION 2: The F-decomposition is INCIDENCE-INDEPENDENT.')
else
   fprintf('CONCLUSION 2: The F-decomposition is NOT INCIDENCE-INDEPENDENT.')
end
fprintf('\n')
if sum==s & sum2==diff
   fprintf('CONCLUSION 3: The F-decomposition is BI-INDEPENDENT.')
else
   fprintf('CONCLUSION 3: The F-decomposition is NOT BI-INDEPENDENT.')
end
\end{lstlisting}
}

\end{document}